\documentclass[12pt]{amsart}
\usepackage{color}
\usepackage{graphicx}
\usepackage{amssymb}
\usepackage{amsthm}
\usepackage{amsmath}
\usepackage[margin=1in]{geometry}
\usepackage{caption}
\usepackage{subcaption}
\usepackage{float}

\newtheorem{theorem}{Theorem}
\newtheorem{lemma}{Lemma}
\newtheorem{proposition}{Proposition}

\newcommand{\bea}{\begin{eqnarray*}}
\newcommand{\eea}{\end{eqnarray*}}
\newcommand{\ben}{\begin{eqnarray}}
\newcommand{\een}{\end{eqnarray}}
\newcommand{\beq}{\begin{equation}}
\newcommand{\eeq}{\end{equation}}

\newcommand{\supp}{\operatorname{supp}}

\begin{document}

\title{Analysis of a singular Boussinesq Model}
\author{Alexander Kiselev}
\email{kiselev@math.duke.edu}
\address{Duke University
Department of Mathematics
Durham, NC 27705}
\author{Hang Yang}
\email{hang.yang@rice.edu}
\address{Rice University
Department of Mathematics -- MS 136
P.O. Box 1892
Houston, TX 77005-1892}

\begin{abstract}
Recently, a new singularity formation scenario for the 3D axi-symmetric Euler equation and the 2D inviscid Boussinesq system has been proposed by Hou and Luo based on extensive numerical simulations \cite{HouLuo1,HouLuo2}.
As the first step to understand the scenario, models with simplified sign-definite Biot-Savart law and forcing
have recently been studied in \cite{CHKLSY,CKY,DKX,hyperbolicBoussinesq,HouLiu,KiselevTan}. In this paper,
we aim to bring back one of the complications encountered in the original equation - the sign changing kernel in the Biot-Savart law.
This makes analysis harder, as there are two competing terms in the fluid velocity integral whose balance determines the regularity properties of the solution.
The equation we study here is based on the CKY model introduced in \cite{CKY}. We prove that finite time blow up persists in a certain range of parameters.
\end{abstract}
\maketitle
\section{introduction}
The 2D inviscid Boussinesq system in vorticity form is given by
\begin{align}
\partial_t &\omega +u\cdot \partial_x \omega =\rho_{x_1} \label{Bouss1}\\
\partial_t &\rho +u\cdot \partial_x \rho =0 \label{Bouss2}\\
& u=\nabla^{\perp}(-\Delta)^{-1}\omega. \label{Bouss3}
\end{align}
The system models ideal fluid driven by buoyancy force \cite{Gill,Pedlosky}.
Solutions to the 2D Boussinesq system are globally regular if the dissipative terms $\Delta \omega,$ $\Delta \rho$ are present in at least one of the equations \eqref{Bouss1}, \eqref{Bouss2}
respectively \cite{Chae,HouLi}.  Models with fractional and/or partial diffusion have also been considered in \cite{AdCauWu1,AdCaoWu2,Wu2,Wu3,Wu4,Wu5}, where the authors show global regularity under various conditions and constraints. In the inviscid case, the finite time blow up vs global regularity question is open; in particular, it appears on the
"Eleven Great Problems of Mathematical Hydrodynamics" list by Yudovich \cite{Yudovich}. Also, the 2D inviscid Boussinesq system is very similar to the 3D axi-symmetric Euler equation
away from the symmetry axis \cite{MajdaBertozzi}. In particular, the presence of $\rho_{x_1}$ on the right hand side of (\ref{Bouss1}) enacts vortex stretching which is a common trait among the hardest problems of mathematical fluid mechanics, e.g. 3D Euler equations and 3D Navier-Stokes equations.


A few years ago, Hou and Luo \cite{HouLuo1} investigated numerically a new possible blow up scenario for the 3D axi-symmetric Euler equation. Their set up involves an infinite height cylinder
with no penetration boundary conditions on the cylinder boundary and periodic boundary conditions in the vertical direction. The initial data $\omega^\theta$ is zero and $u^\theta$ is odd with respect to
the $z$ variable. Rapid growth of vorticity $\omega^\theta$ is observed at a ring of hyperbolic points of the flow along the boundary in the $z=0$ plane \cite{HouLuo1,HouLuo2}.
For the 2D Boussinesq system, the scenario involves (after $\pi/2$ rotation) an infinite horizontal strip with solutions periodic in $x_1$ and satisfying no penetration condition on the
strip boundary. In the scenario, $\omega$ is odd and $\rho$ is even with respect to $x_1.$ Very fast growth of $\omega$ is observed at a hyperbolic point of the flow located at $x_1=0$ on the strip boundary. It should be noted that there is evidence that hyperbolic points of the flow play an important role in a number of important fluid mechanics phenomena. In particular, a recent experimental paper \cite{expturb} shows that most instances of extreme dissipation in a turbulent flow happen in regions featuring hyperbolic point/front type local geometry of the flow.

Motivated by the Hou-Luo scenario, Kiselev and Sverak \cite{KiselevSverak} considered 2D Euler equation - obtained by setting $\rho=0$ in \eqref{Bouss2} - in a similar geometry.
They constructed an example of a smooth solution with double exponential in time growth of the gradient of vorticity, showing that the upper bounds on growth of the derivatives of $\omega$
available since 1930s are qualitatively sharp.

A 1D model of the Hou-Luo scenario has been proposed already in \cite{HouLuo1}. Several works have analyzed this and a few other related models, in all cases proving finite time singularity formation
\cite{CHKLSY,CKY,DKX,HouLiu}. All these models feature Biot-Savart laws $u(x,t) = -\int_0^\infty K(x,y) \omega(y,t)\,dy$ with non-negative kernels $K.$
This helps prove transport of vorticity and density towards the origin, accompanied by growth in $\rho_{x_1}$ leading to growth of vorticity and thus to nonlinear feedback
feedback loop driving blow up.

The first two-dimensional models of the Hou-Luo scenario have been considered in \cite{hyperbolicBoussinesq,KiselevTan}. Both models are set in the first quadrant of the plane (implicitly
assuming odd symmetry of the solution) and are given by
\begin{align}
  \partial_t \omega+u\cdot \nabla \omega=\frac{\rho}{x_1} \label{modifiedBouss1}\\
\partial_t \rho+u\cdot \nabla \rho =0 \label{modifiedBouss2}\\
u(x,t)=\left(-x_1 \Omega(x,t), \,\,\, x_2 \Omega(x,t)\right). \label{hyperbolicBoussBS}
\end{align}
The models differ in the choice of $\Omega:$ in \cite{hyperbolicBoussinesq}
\[ \Omega(x,t)= \int_{S_{\alpha}}\frac{\omega(y,t)}{|y|^2}dy, \]
where $S_{\alpha}=\{(x_1,x_2):0<x_1,0<x_2<\alpha x_1\}$ is a sector in the first quadrant with arbitrary large $\alpha$ as a parameter. In \cite{KiselevTan} a slightly different integration domain $D=\{(y_1,y_2):y_1y_2\geq x_1x_2\}$ is chosen in the definition of $\Omega$. The choice of the $\Omega$ in \cite{KiselevTan} leads to incompressible fluid velocity, while the velocity
in \cite{hyperbolicBoussinesq} is not incompressible but is closer in form to the velocity representation for the 2D Euler solutions established in \cite{KiselevSverak}.
Also, both models use simplified mean field forcing term $\rho/x_1,$ which ensures that vorticity has fixed sign. The initial data is taken smooth, and supported away from $x_1$ axis.
In both works, finite time blow up is established for a fairly broad class of initial data.

Both of the above mentioned modifications as well as all 1D models considered so far share the same feature that particle trajectories for positive vorticity solutions always point to one direction: towards the $x_1=0$ axis. However, in the true 2D Boussinesq system, the kernel in the Biot-Savart law is not sign definite.
The fluid velocity is given, in a half plane $x_2 \geq 0$ and under the odd in $x_1$ symmetry assumption on $\omega,$ by
\begin{eqnarray}\nonumber
u_1(x,t) = \frac{1}{2\pi} \int_0^\infty \int_0^\infty \left(\frac{x_2-y_2}{(x_1-y_1)^2 +(x_2-y_2)^2}-\frac{x_2-y_2}{(x_1+y_1)^2 +(x_2-y_2)^2}- \right. \\
\left. \frac{x_2+y_2}{(x_1-y_1)^2 +(x_2+y_2)^2}+\frac{x_2+y_2}{(x_1+y_1)^2 +(x_2+y_2)^2}    \right) \omega(y,t)dy_1dy_2.\label{bsbous}
\end{eqnarray}
The second component $u_2$ is given by a similar formula. It is not hard to see that in \eqref{bsbous} the kernel is positive on the part of integration region, and positive vorticity
in these regions works against blow up.

In this paper, we propose a 1D model set on $\mathbb{R}$ given  by
\begin{align}
\partial_t \omega &+u\cdot \partial_x \omega =\frac{\rho}{x} \label{simplifiedOmega}\\
\partial_t \rho &+u\cdot \partial_x \rho =0 \label{simplifiedRho}\\
u(x,t) &=x\int_{{\rm min}(\beta_2x,1)}^{{\rm min}(\beta_1x,1)} \frac{\omega(y,t)}{y}dy-x\int_{{\rm min}(\beta_1x,1)}^{1} \frac{\omega(y,t)}{y}dy\label{nonDefiniteBS}\\
\omega(x,0)&=\omega_0 (x), \,\,\, \rho(x,0)=\rho_0(x) \label{initial}
\end{align}
where $0<\beta_2\leq 1\leq \beta_1<\infty$ are two prescribed parameters.
Note that we effectively limit the meaningful evolution to $(0,1)$ interval since we are interested in dynamics near zero.
Extending integration in the Biot-Savart law beyond $1$ does not add anything essential to the model since the kernel is regular
in the added region, but leads to some technicalities associated with estimating growth of support of $\omega, \rho.$
In what follows below, for the sake of notational simplicity, we will omit the min condition in the limits
of Biot-Savart integral. It should be always understood that if $\beta_{1,2}x \geq 1,$ the integral limits are cut off at $1.$

The model
is close to the CKY model of \cite{CKY}, which can be obtained by setting $\beta_1=\beta_2$ in \eqref{nonDefiniteBS}
and replacing $\rho/x_1$ with $\rho_{x_1}.$
The ``anti blow up" region is $(\beta_1 x,\beta_2 x)$ and it is the part of the integration region closest to $x=0,$ a feature that is also shared by \eqref{bsbous}.
This region also tends to include the largest values of vorticity, making the overall balance highly nontrivial.
The main purpose of this paper is to begin to assemble the technical tools needed for analysis of models with more complex Biot-Savart relationships,
with the eventual goal of getting insight into the workings of the true Biot-Savart laws appearing in the key equations of fluid mechanics such as the 2D Boussinesq system
or the SQG equation.

To set up local well-posedness theory, we will follow \cite{KiselevTan} and use the space $K^n$ of compactly supported in $(0,1)$ functions defined by the norm
$$\|f\|_{K^n}:=\|f\|_{C^n}+(\min_x\{\supp(f)\})^{-1}.$$ Here $n$ is an integer.
This space is well adapted to the mean field forcing term in \eqref{simplifiedOmega}.
We also denote
$K^\infty=\bigcap_{n\geq 1}K^n$.
\begin{theorem}\label{localExistence2}
Given non-negative initial data $(\omega_0,\rho_0)\in K^n((0,1))\times K^n((0,1))$, $n \geq 1,$ there exists $T=T(\omega_0,\rho_0)$ such that the system (\ref{simplifiedOmega})-(\ref{initial}) has a local-in-time unique solution $(\omega,\rho)\in C([0,T],K^n((0,1))\times C([0,T),K^n((0,1)))$.
\end{theorem}
\it Remark. \rm It is not difficult to remove the non-negative assumption on the initial data. We do not pursue the most general case here since proving
finite time singularity formation is our main objective. \\

\begin{theorem}\label{NonDefiniteBlowUp}
Assume that $\beta_2\leq1\leq\beta_1<2\beta_2$. There exist compactly supported $(\omega_0,\rho_0) \in K^\infty((0,1))\times K^\infty((0,1))$ such that the corresponding solution of
(\ref{simplifiedOmega})-(\ref{initial}) blows up in finite time in the sense that  $$\int_0^{T}\|\omega(\cdot,t)\|_{L^\infty}dt=\infty,\qquad \int_0^{T}\|\partial_x\rho(\cdot,t)\|_{L^\infty}dt=\infty, \qquad \int_0^{T}\|\partial_x u(\cdot,t)\|_{L^\infty}dt=\infty$$
for some $T\in(0,\infty)$.
\end{theorem}

\it Remark. \rm
The assumption $\beta_2 \leq 1 \leq \beta_1<2\beta_2$ is necessary for the current argument to yield finite time singularity formation.
It seems likely that this condition is not sharp, but new ideas are needed to improve the blow up parameter range.

\section{Local well-posedness and continuation criteria}\label{lwp}

The proof of the local existence Theorem \ref{localExistence2}
for the model (\ref{simplifiedOmega})-(\ref{initial}) can be carried out with essentially the same argument as in \cite{KiselevTan}, so we will provide just a sketch of the proof for the sake of brevity.
The key is to control the distance from the support of the solution to the origin. It is not hard to see that while this distance remains positive, the system
(\ref{simplifiedOmega})-(\ref{initial}) has well controlled forcing and Biot-Savart law, and the solutions retain original regularity.

Denote this distance by
$$\delta(t):=\min_x\{\supp(\omega)\cup\supp(\rho). \}$$
The next lemma explains how we can bound $\delta(t)$ away from zero for at least a short period of time.
This is an a-priori estimate; to properly show local existence of solutions and the associated bounds one needs to use an iterative approximation
scheme similar to \cite{KiselevTan}.
Let $\Phi(x,t)$ be particle trajectories defined as usual by
\begin{equation}\label{trajectories} \partial_t \Phi(x,t) = u(\Phi(x,t), t), \,\,\,\Phi(x,0)=x. \end{equation}
\begin{lemma}\label{lemma1}
Suppose that $\omega_0, \rho_0$ are as in the assumption of Theorem~\ref{localExistence2}, and let $\omega, \rho \in C(K_n,[0,T])$ solve (\ref{simplifiedOmega})-(\ref{initial}).
Write $$\Psi(x,t)=\sup_{s\leq t}\log(1/\Phi(x,s)).$$
Then $\Psi(x,t)$ satisfies
\begin{equation}\label{upperpsibound}
\partial_t \Psi(x,t) \leq C \Psi(x,t)(1+t e^{\Psi(x,t)}), \,\,\, \Psi(x,0)= \log x^{-1}.
\end{equation}
Therefore, there exists $T>0$ such that for $0 \leq t \leq T,$ $\Psi(x,t)$ remains finite for all $x \in {\rm supp}(\omega_0, \rho_0).$
\end{lemma}
\begin{proof}
Solving the equations along trajectories $\Phi$ defined in (\ref{trajectories}) we obtain
\begin{equation} \label{alongTrajectories1}
\rho(x,t)=\rho_0(\Phi^{-1}(x,t)),\quad \omega(x,t)=\omega_0(\Phi^{-1}(x,t))+\rho_0(\Phi^{-1}(x,t))\int_{0}^t \frac{1}{\Phi(\Phi^{-1}(x,t),s)}ds
\end{equation}
This in particular indicates preservation of non-negavity of $\rho$ and $\omega$ by the evolution.

Due to positivity of $\omega$ and $\beta_1 \geq 1$ we have
$$\frac{d}{dt}\Phi(x,t) \geq-\Phi(x,t) \int_{\Phi(x,t)}^{1}\frac{\omega(y,t)}{y}dy\,\,\,\Longrightarrow \,\,\,
\frac{d}{dt} \log(1/\Phi(x,t))\leq\int_{\Phi(x,t)}^{1}\frac{\omega(y,t)}{y}dy.$$


Now by (\ref{alongTrajectories1})
\begin{equation}
\omega(y,t)\leq \|\omega_0\|_{L^\infty}+\|\rho_0\|_{L^\infty}\int_{0}^{t} \frac{1}{\Phi(\Phi^{-1}(y,t),s)}ds\leq C\bigg(1+\int_{0}^{t} \frac{1}{\Phi(\Phi^{-1}(y,t),s)}ds\bigg)
\end{equation}
Also, if $y\in [\Phi(x,t),1]$, then $\Phi^{-1}(y,t)\in[x,1]$. Since the trajectories cannot cross while solution remains regular, we have
$$\frac{1}{\Phi(\Phi^{-1}(y,t),s)}\leq\frac{1}{\Phi(x,s)}\leq e^{\Psi(x,t)}$$
Therefore
$$\partial_t \Psi(x,t)\leq C\int_{\Phi(x,t)}^{1} \frac{1}{y}\bigg(1+ \int_{0}^{t}e^{\Psi(x,t)}ds\bigg)dy \leq C\Psi(x,t)(1+te^{\Psi(x,t)})$$
yielding \eqref{upperpsibound}.
\end{proof}
Note that $\delta(t) = e^{-\Psi(\delta(0),t)},$ so Lemma~\ref{lemma1} allows control of $\delta(t)$ for $t \leq T$ (and so implies regularity of the solution).

The proposition that we prove next is an analogue of the well-known result due to Beale-Kato-Majda \cite{BKM}. It will provide continuation criteria for solutions.

\begin{proposition} \label{BKMcriterion}
Let $n \geq 1$ be an integer. The following are equivalent:\\
(a) The solution $(\omega,\rho)\in C([0,T),K^n)\times C([0,T),K^n)$ can be continued past $T$ \\
(b) $\int_{0}^T \|\partial_x u(\cdot,t)\|_{L^\infty}dt<\infty$\\
(c) $\int_{0}^T \|\partial_x\rho(\cdot,t)\|_{L^\infty}dt<\infty$\\
(d) $\int_{0}^T \|\omega(\cdot,t)\|_{L^\infty}dt<\infty$\\
(e) $\liminf_{t \rightarrow T} \delta(t)>0.$
\end{proposition}
\begin{proof}
The equivalence of (a) and (e) follows from the definition of the norm $K^n$ and the above discussion on how positive $\delta(t)$ ensures local existence of solution
in $K^n$ on a time interval depending only on the size of $\delta.$
In fact, (e) implies all other conditions in the lemma by the argument mentioned above: the solution supported away from $x=0$ uniformly in a given time interval maintains regularity
by straightforward estimates.

Equivalence between (a) and (b) can be obtained through a standard argument based on the Lagrangian formulation of the system.
Note that we only need to show (b) implies (a). A standard estimate on the trajectories, using the fact that the origin is a fixed point of the flow, yields
\begin{equation}\label{graducon} \delta'(t) = \frac{d}{dt} \Phi(\delta(0),t) \geq -\|\partial_x u(\cdot,t)\|_{L^\infty} \delta(t). \end{equation}
Thus by Gronwall, \[ \delta(t) \geq \delta(0)\exp \left(-\int_0^T \|\partial_x u (\cdot, t)\|_{L^\infty}\,dt\right). \]

To prove the implication $(b)\Rightarrow (c)$ , differentiate (\ref{simplifiedRho}) and compose with $\Phi$ to get
$$\frac{d}{dt}\partial_x\rho(\Phi(x,t),t)=-\partial_xu(\Phi(x,t),t)\partial_x\rho(\Phi(x,t),t)$$
Thus $$\frac{d}{dt}\|\partial_x \rho\|_{L^\infty}\leq \|\partial_x u\|_{L^\infty}\|\partial_x\rho\|_{L^\infty}$$
to which we can again apply Gr\"onwall's inequality and establish $(b)\Rightarrow (c)$.

The implication $(c)\Rightarrow (d)$ follows  from integrating (\ref{simplifiedOmega}) in Lagrangian coordinates and estimating
(using $\rho(0,t)=0$, before blow up)
$$|\omega(\Phi(x,t),t)|=\bigg|\omega_0(x)+\int_{0}^{t}\frac{\rho(\Phi(x,s),s)}{\Phi(x,s)}ds\bigg| \leq \|\omega_0\|_{L^\infty}+\int_{0}^{t}\frac{\|\partial_x\rho(\cdot,s)\|_{L^\infty}\cdot |\Phi(x,s)|}{|\Phi(x,s)|}ds$$
for all $x$.

To show $(d)\Rightarrow(e)$, assume solution exists up to $T$ and $\int_{0}^T \|\omega(\cdot,t)\|_{L^\infty}dt=M$. 
Observe that differentiating \eqref{nonDefiniteBS} we obtain
\begin{align} \label{gradUbyOmega}
|\partial_x u(\Phi(x,t),t)| \leq \left|\int_{\beta_2\Phi(x,t)}^{1}\frac{\omega(y,t)}{y}dy\right| +C\|\omega\|_{L^\infty}&\leq C\|\omega(\cdot,t)\|_{L^\infty}(1+|\log \Phi(x,t)|),
\end{align}
where $C$ depends only on $\beta_{1,2}.$
Taking $x=\delta(0)$ and Combining (\ref{graducon}) and (\ref{gradUbyOmega}), we see that
$$\frac{d}{dt}\delta(t)\geq -C\|\omega(\cdot,t)\|_{L^\infty}\delta(t)(1+\log(1/\delta(t)))$$
So \[ \log \delta(t)^{-1} \leq \log \delta(0)^{-1} e^{C\int_0^t \|\omega(\cdot, s)\|_{L^\infty}\,ds} + \left(e^{C\int_0^t \|\omega(\cdot, s)\|_{L^\infty}\,ds}-1 \right), \]
finishing the proof.
\end{proof}

\section{Warming-up: A special case with sign-definite Biot-Savart law}

As a warm-up, let us first take a look at a special case of the model (\ref{simplifiedOmega})-(\ref{initial}) by further simplifying the Biot-Savart law. Take $\beta_1=\beta_2=1$ and consider
the following model on unit interval $[0,1]:$
\begin{align}
\partial_t \omega &+u\cdot \partial_x \omega =\frac{\rho}{x}\label{specialCase1} \\
\partial_t \rho &+u\cdot \partial_x \rho =0\label{specialCase2} \\
u(x,t) &=-x\int_{x}^{1} \frac{\omega(y,t)}{y}dy\label{negativeBS}\\
\omega(x,0)&=\omega_0 (x), \rho(x,0)=\rho_0(x) \label{specialCaseInitial}
\end{align}
The model then becomes close to the CKY model, but even easier due to simpler forcing term. The proof of blow up is very transparent.



\begin{theorem}\label{DefiniteBlowUp}
There exists $(\omega_0,\rho_0) \in K^\infty((0,1))\times K^\infty((0,1))$ such that the corresponding solution of (\ref{specialCase1})-(\ref{specialCaseInitial}) blows up in finite time in the sense that  $$\int_0^{T}\|\omega(\cdot,t)\|_{L^\infty}dt=\infty,\qquad \int_0^{T}\|\partial_x\rho(\cdot,t)\|_{L^\infty}dt=\infty, \qquad \int_0^{T}\|\partial_x u(\cdot,t)\|_{L^\infty}dt=\infty$$
for some $T\in(0,\infty)$.
\end{theorem}
\begin{proof}
Denote $I=(0,1).$ Consider $\rho_0\in C^{\infty}_0(I)$ such that $0\leq \rho_0 \leq 1$ and $\rho_0\equiv 1$ on $[\frac13,\frac23]$.
For simplicity, choose $\omega_0=0.$ 

The idea is to control how the support of $\rho_0$ moves towards the origin. We assume that the solution stays regular and show that
the characteristics originating at the points with nonzero $\rho_0$ arrive at  the origin in finite time, thus implying that $\delta(t)$ becomes zero
in finite time and then all other blow up characterizations of Proposition~\ref{BKMcriterion} hold.

Note that since $\rho$ and $\omega$ are nonnegative, trajectories always move in the negative $x$ direction. Compute
\begin{align*}
\frac{d^2}{dt^2}\log\left(\frac{1}{\Phi(x,t)}\right)& = -\frac{d\Phi(x,t)}{dt}\cdot \frac{\omega(\Phi(x,t),t)}{\Phi(x,t)}+\int_{\Phi(x,t)}^{1}\frac{-u\partial_x\omega+\frac{\rho}{y}}{y}dy \\
&= \omega(\Phi(x,t),t)\int_{\Phi(x,t)}^{1} \frac{\omega(y,t)}{y}dy -\frac{u\omega}{y}\bigg|_{\Phi(x,t)}^1+ \int_{\Phi(x,t)}^{1} \frac{\omega^2(y,t)}{y}dy+\int_{\Phi(x,t)}^{1} \frac{\rho(y,t)}{y^2}dy\\
&=\int_{\Phi(x,t)}^{1} \frac{\omega^2(y,t)}{y}dy+\int_{\Phi(x,t)}^{1} \frac{\rho(y,t)}{y^2}dy\\
& \geq \int_{\Phi(x,t)}^{1} \frac{\rho(y,t)}{y^2}dy = \int_{\Phi(x,t)}^{1} \frac{\rho_0(\Phi^{-1}(y,t))}{y^2}dy
\end{align*}
Also
$$\int_{\Phi(\frac13,t)}^{1} \frac{\rho_0(\Phi^{-1}(y,t))}{y^2}dy\geq \int_{\Phi(\frac13,t)}^{\Phi(\frac23,t)}\frac{1}{y^2}=\frac{1}{\Phi(\frac13,t)}-\frac{1}{\Phi(\frac23,t)}.$$
 Moreover, (\ref{trajectories}) and (\ref{negativeBS}) together also imply that
$$\frac{d}{dt}\log\left(\frac{1}{\Phi(\frac13,t)}\right)\geq\frac{d}{dt}\log\left(\frac{1}{\Phi(\frac23,t)}\right)$$
leading to $$\log\left(\frac{1}{\Phi(\frac13,t)}\right)-\log\left(\frac{1}{\Phi(\frac23,t)}\right)\geq \log 2,$$
or equivalently $$\frac{1}{\Phi(\frac13,t)}\geq \frac{2}{\Phi(\frac23,t)}$$
Combining all of the above, we have
\begin{equation} \label{differentialIneq1}
\frac{d^2}{dt^2}\log \bigg(\frac{1}{\Phi(\frac13,t)}\bigg)\geq \int_{\Phi(\frac13,t)}^{1} \frac{\rho_0(\Phi^{-1}(y,t))}{y^2}dy\geq \frac{1}{\Phi(\frac13,t)}-\frac{1}{\Phi(\frac23,t)}\geq
\frac{1}{2\Phi(\frac13,t)}.
\end{equation}
Write $y(t)= 1/\Phi(\frac13,t)$. Then based on \eqref{differentialIneq1} we have $y(t) \geq G(t),$ where
\begin{equation}\label{Geq}
G''(t) = \frac12 G(t),\,\,\,G(0)=1, \,\,\,G'(0)=0.
\end{equation}
The choice of the initial condition for the derivative in \eqref{Geq} follows from $y'(0)=0,$ which is a consequence
of our choice $\omega_0(x) \equiv 0.$ Finite time blow up for $G$ is not hard to establish.
Introduce a new variable $v=G'$ and observe that
\begin{equation}\label{comp171}
\frac{1}{6}\frac{d(G^3)}{dG}=\frac{1}{2}G^2=G''=v'=\frac{dv}{dG}G'
=\frac{1}{2}\frac{d(v^2)}{dG}\Longrightarrow v^2=\frac{1}{3}(G^{3}-(\log 3)^3).
\end{equation}
Then
$$v'=G''=\frac{1}{2}G^2=(3v^2+(\log 3)^3)^{2/3}\geq C v^{4/3}, \quad v(0)=G'(0)=0.$$
From this, we can deduce that \[ v(t) \geq \frac{v(0)}{(1-C v(0)^{1/3}t)^3} \]
and thus $v(t)$ and also, according to \eqref{comp171}, $G(t)$ blow up in finite time.
\end{proof}

\section{The Model with non-sign-definite Biot-Savart law}

To study the non-sign-definite model, it will be convenient to introduce a change of variable $z=-\log x$. Denote $\tilde{\rho}(z,t)=\rho(x(z),t)$, $\tilde{\omega}(z,t)=\omega(x(z),t)$ and $\tilde{u}(z,t)=-x(z)^{-1} u(x(z),t)$. In the $z-$coordinate, equations (\ref{simplifiedOmega}), (\ref{simplifiedRho}) and (\ref{nonDefiniteBS}) take form
\begin{align}
\partial_t \tilde{\omega} &+\tilde{u}\cdot \partial_z \tilde{\omega} =\tilde{\rho}\cdot e^z \label{omegaInZ}\\
\partial_t \tilde{\rho} &+\tilde{u}\cdot \partial_z \tilde{\rho} =0 \label{rhoInZ}\\
\tilde{u}(z,t) &=\int_{0}^{z-\gamma_1}\tilde{\omega}(y,t)dy-\int_{z-\gamma_1}^{z+\gamma_2}\tilde{\omega}(y,t)dy \label{BSinZ}
\end{align}
where $\gamma_1=\log \beta_1, \gamma_2=\log \beta_2^{-1}$ and $2e^{-\gamma_1}-e^{\gamma_2}>0$ due to our assumptions on $\beta_{1,2}.$  \\

We will work with the model (\ref{omegaInZ})-(\ref{BSinZ}) for the rest of the paper and abuse notation to suppress tilde and write $(\omega,\rho, u)$  as the solution to (\ref{omegaInZ})-(\ref{BSinZ}) instead of $(\tilde{\omega}, \tilde{\rho},\tilde{u})$. We will also abuse notation to denote $\Phi$ the particle trajectories defined by $\tilde{u}$ via $(\ref{trajectories})$.
Note that in the $z$ formulation, the blow up condition $\delta(t) \rightarrow 0$ becomes $\Phi(Z,t) \rightarrow \infty$ for $Z = {\rm sup} ({\rm supp}(\omega_0, \rho_0)).$

Unlike the method we used previously on the warm-up model, the blow up of the full model becomes more delicate.
It is conceivable that the negative contribution in \eqref{BSinZ} arrests propagation of trajectories to infinity, especially since the negative contribution comes from
the largest $z$ in the support of solution where we can expect $\omega$ to be largest due to the forcing term \eqref{omegaInZ}. We will need to establish a sort of monotonicity
structure that allows to prove blow up. The argument will focus on growth of $\partial_z \Phi(z,t).$


\textbf{The Choice of Initial Data and Parameters}\\
For the rest of the paper, we fix $\beta_1, \beta_2$ (which in succession fixes $\gamma_1, \gamma_2$ respectively) and $\epsilon$ small enough such that
\begin{equation} \label{defEpsilon}
e^{-\gamma_1-\gamma_2-\epsilon}-1>0.
\end{equation}
Next, let the parameters $L_{0,1,2,3,4}$ have the ordering $1<L_0<L_1<L_2<L_3<L_4$. Fix $L_{0},L_1$ such that $L_0\leq L_1/4,$ $\gamma_{1,2}<L_1/4,$ and $\epsilon <L_1/10.$ The choice of $L_{2},L_3$ will be specified later and $L_4$ will be fixed with only one constraint $L_4>L_3$ once $L_3$ is chosen. The initial data $\omega_0, \rho_0$ will be constructed as follows: $\omega_0= 0$ for simplicity; $\rho_0\in C_0^\infty$ is supported on $[1,L_4]$ and such that $0\leq \rho\leq 1$, $\rho_0([L_0,L_3])=1$ and $\rho_0$ is monotone decreasing for $z>L_3$. 

Let us start with a useful a-priori bound on $\Phi$, which is just a $z$-variant of Lemma~\ref{lemma1}.
\begin{lemma}  \label{GammaBound}
Take $\rho_0,\omega_0$ as above. Let $\Gamma(z,t)$ be the solution to
\begin{equation}\label{gammaeq} \partial_t\Gamma(z,t)=e^{\Gamma(z,t)} \cdot\Gamma(z,t)\cdot t,\quad \Gamma(z,0)=z. \end{equation}
Then we have $\Phi(z,t)\leq \Gamma(z,t)$ for all $z$ for as long time as $\Gamma$ is defined.
\end{lemma}
\begin{proof}
Local existence of $\Gamma$ follows by Picard's Theorem. Write $\Psi(z,t):=\sup_{s\in[0,t]}\Phi(z,s)$. Along the particle trajectories $\Phi(z,t)$, we now have
\begin{equation} \label{alongTrajectories2}
\rho(z,t)=\rho_0(\Phi^{-1}(z,t)),\quad \omega(z,t)=\rho_0(\Phi^{-1}(z,t))\int_{0}^{t}e^{\Phi(\Phi^{-1}(z,t)),s)}ds
\end{equation}
So $\omega$ and $\rho$ remain non-negative if they are non-negative initially. Then, given $\omega_0=0$, we have
\begin{align*}
\partial_t\Psi(z,t)&\leq \int_{0}^{\Phi(z,t)}\omega(y,t)dy \leq \int_{0}^{\Phi(z,t)}\int_{0}^{t}e^{\Phi(\Phi^{-1}(y,t)),s)}dsdy
\end{align*}
Here we have used $0\leq\rho_0\leq 1$ and $\omega\geq0$. For $y$ in the integration domain $[0,\Phi(z,t)]$, one must have $\Phi(\Phi^{-1}(y,t),s)\leq \Phi(z,s)\leq \Psi(z,s)$. This is due to non-crossing of trajectories, i.e. $\Phi(z_1,t)\leq\Phi(z_2,t)$ for all $t$ if $z_1\leq z_2$ and similarly for the inverse trajectories $\Phi^{-1}$. Thus, we continue to estimate and arrive at
$$\partial_t\Psi(z,t)\leq \int_{0}^{\Psi(z,t)}\int_{0}^{t}e^{\Psi(z,s)}dsdy\leq \Psi(z,t)\cdot e^{\Psi(z,t)}\cdot t$$
as $\Psi$ is increasing in $t$. A simple comparison $\Phi(z,t)\leq \Psi(z,t)\leq \Gamma(z,t)$ completes the proof.
\end{proof}

A key quantity that we will need to estimate is \[ \frac{\partial_z u(\Phi(z,t),t)}{\partial_z \Phi(z,t)} = 2\omega(\Phi(z,t)-\gamma_1,t)-\omega(\Phi(z,t)+\gamma_2,t). \]
The first step is showing that this quantity becomes positive on most of the support of $\rho_0$ for a very short initial time. This would imply that in this range,
$\partial_z \Phi(z,t)$ is initially growing. One can think of this estimate as a sort of establishment of induction base, to be followed by ``induction step".

\begin{lemma} \label{positivityLemma1}
There exists $t_0=t_0(L_1,L_4)$ such that for all $0<t \leq t_0$ and $L_1\leq z< \Phi^{-1}(\Phi(L_4,t)+\gamma_1,t)$, we have
\begin{equation}\label{keyomexp} 2\omega(\Phi(z,t)-\gamma_1,t)-\omega(\Phi(z,t)+\gamma_2,t)>0. \end{equation}
Note also that the expression in \eqref{keyomexp} is zero for any $z \geq \Phi^{-1}(\Phi(L_4,t)+\gamma_1,t)$ and any time $t$
while solution exists.
\end{lemma}
\begin{proof}
Note that with the initial data $\rho_0$ described above, the local existence time is controlled by finiteness of the solution to \eqref{gammaeq} corresponding to
the value $z=L_4.$ Then
by local existence and continuity of $\Phi(L_1,t)$, as well as the assumption $L_0 \leq L_1/4,$ there exists a short time $T_0$ such that the solution stays regular and
\begin{equation}\label{Phicontest} \Phi(L_0,t) \leq L_1/2 \,\,\,{\rm and}\,\,\, \Phi(L_1,t)\geq \frac{5L_1}{6},\quad \forall t\leq T_0. \end{equation}
Then we must have
\begin{equation} \label{rightToL1}
\Phi(z,t)-\gamma_1\geq \Phi(L_1,t)-\gamma_1\geq L_1/2,\forall z\geq L_1, \forall t\leq T_0
\end{equation}
because $\gamma_1<L_1/4$; otherwise, trajectories will cross. Denote
$$z_{-}(t)=\Phi^{-1}(\Phi(z,t)-\gamma_1,t),\qquad z_{+}(t)=\Phi^{-1}(\Phi(z,t)+\gamma_2,t).$$
Of course $z_{\pm}$ depend on $z$ but we will suppress this in notation.
Now notice that if $z\in[L_1, \Phi^{-1}(\Phi(L_4,t)+\gamma_1,t)),$ then by \eqref{Phicontest} we have $L_0 \leq z_{-}(t)<L_4$ if $t \leq T_0$ and
therefore $\rho_0(z_{-}(t))>0$. We utilize (\ref{alongTrajectories2}) and monotonicity of $\rho_0$ in the region $z>L_0$ to get
\begin{align}
2\omega(\Phi(z,t)-\gamma_1,t)-\omega(\Phi(z,t)+\gamma_2),t) &=2\rho_0(z_{-}(t))\int_{0}^{t}e^{\Phi(z_{-}(t),s)}ds-\rho_0(z_{+}(t))\int_{0}^{t}e^{\Phi(z_{+}(t),s)}ds \nonumber \\
&\geq \rho_0(z_{-}(t))\int_{0}^{t}\big(2e^{\Phi(z_{-}(t),s)}-e^{\Phi(z_{+}(t),s)}\big)ds,  \label{contra}
\end{align}
for all $z\in[L_1,\Phi^{-1}(\Phi(L_4,t)+\gamma_1,t)).$
Now fix $t_0<T_0$ such that for every $z\in[L_1, \Phi^{-1}(\Phi(L_4,t)+\gamma_1,t))$
\begin{equation}
\Phi(z_{-}(t),s)\geq z-\gamma_1-\epsilon_1, \quad \Phi(z_{+}(t),s)\leq z+\gamma_2+\epsilon_2 \label{twoInequalities}
\end{equation}
for all $0\leq s\leq t\leq t_0$ with $\epsilon_{1,2}>0,$ $\epsilon_1+\epsilon_2\leq \epsilon$ (see definition of $\epsilon$  in (\ref{defEpsilon})). Such $t_0$ can be found due to local existence,
continuity of $\Phi(z,t),$ and finiteness of the domain. Therefore, plugging $(\ref{twoInequalities})$ into $(\ref{contra})$ yields that for $t\leq t_0$ and
$z\in[L_1, \Phi^{-1}(\Phi(L_4,t)+\gamma_1,t))$ we have
\begin{align*}
2\omega(\Phi(z,t)&-\gamma_1,t)-\omega(\Phi(z,t)+\gamma_2),t)\geq \rho_0(z_{-}(t))\int_{0}^{t}\big(2e^{z-\gamma_1-\epsilon_1}-e^{z+\gamma_1+\epsilon_1}\big)ds\\
 &\geq \rho_0(z_{-}(t))e^{z+\gamma_2+\epsilon_2}(2e^{-\gamma_1-\gamma_2-\epsilon}-1)t>0.
\end{align*}
\end{proof}

Let us now outline the plan of the proof of our main result Theorem~\ref{NonDefiniteBlowUp}.
As we already mentioned, Lemma~\ref{positivityLemma1} can be viewed as an ``induction base" - we established positivity of a key quantity for a very short time depending
on ``fast" parameter $L_4.$ In the next proposition, we show that this positivity is preserved, provided that the solution stays regular, for a period of time
that depends only on the ``slow" parameters $L_1,\gamma_{1,2},\epsilon.$ We will then use this positivity to show singularity formation in an arbitrary short time provided that we choose the fast parameter $L_3$ large enough.

Fix $\tau_0>0$ to be such that
\begin{equation}\label{tau0con} \tau_0e^{\Gamma(3L_1,\tau_0)}= \frac{\epsilon}{\tau_0(\gamma_1+\gamma_2+\epsilon)}. \end{equation}
The existence of $\tau_0$ follows from the local bounds on $\Gamma$ evident from \eqref{gammaeq}.
We will also assume that
\begin{eqnarray}\label{tau0con1}
\Phi(L_0,t) \leq \Gamma(L_0,t) < L_0+ \frac{L_1}{6} \leq \frac{L_1}{2} \\
\label{tau0con2}
\Phi(L_1,t) \leq \Gamma(L_1,t) < \frac{3L_1}{2}
\end{eqnarray}
for all $0 \leq t \leq \tau_0$ if necessary by decreasing $\tau_0.$

\begin{proposition} \label{positivityProp1}
For all $t\in [0,\tau_0]$ and while the regular solution exists, and for all $z\in [L_1,\Phi^{-1}(\Phi(L_4,t)+\gamma_1,t))$, we have
\begin{equation}\label{keyomexp1} 2\omega(\Phi(z,t)-\gamma_1,t)-\omega(\Phi(z,t)+\gamma_2,t)>0. \end{equation}
\end{proposition}
\begin{proof}
Suppose not.
Observe that for every $t$ while the solution exists, there is an $\eta(t)>0$ such that for $z\in [\Phi^{-1}(\Phi(L_4,t)+\gamma_1,t)-\eta(t),\Phi^{-1}(\Phi(L_4,t)+\gamma_1,t))$
we must have strict inequality in \eqref{keyomexp1}. This $\eta(t)$ can be determined by condition that
\[ \Phi(\Phi^{-1}(\Phi(L_4,t)+\gamma_1,t)-\eta(t))+\gamma_2 \geq L_4, \]
so that the second term in \eqref{keyomexp1} vanishes. That $\eta(t)>0$ follows from non-intersection of the trajectories while solution stays regular.

Now due to Lemma~\ref{positivityLemma1}, continuity of solution, and compactness of domain
\[ \{ t_0 \leq t \leq \tau_0, \,\,\, L_1 \leq z \leq \Phi^{-1}(\Phi(L_4,t)+\gamma_1,t)-\eta(t) \}, \]
the only way \eqref{keyomexp1} can be violated is if there exists $t_0 < \tau_1 \leq \tau_0$ such that for $t < \tau_1$ \eqref{keyomexp1} holds but
\begin{equation}\label{breaktau1}
2\omega(\Phi(z_1,\tau_1)-\gamma_1,\tau_1)-\omega(\Phi(z_1,\tau_1)+\gamma_2,\tau_1) =0.
\end{equation}
for some $z_1\in [L_1,\Phi^{-1}(\Phi(L_4,\tau_1)+\gamma_1,\tau_1)-\eta(\tau_1)].$

\begin{lemma}\label{decaycon}
We have
\begin{equation}\label{decayest} \Phi(L_1,t)-\gamma_1\geq \Phi(L_0,t), \,\,\,\Phi(2L_1,t) \geq \Phi(L_1,t)+\gamma_2\end{equation} for $0\leq t \leq \tau_1$. \end{lemma}
\begin{proof}
Let us focus on the proof of the first inequality in \eqref{decayest}, as the proof of the second one is similar modulo using \eqref{tau0con2}
instead of \eqref{tau0con1}.
Integrating (\ref{trajectories}) in $t$ and differentiating in $z$ gives
\begin{equation}\label{dzPhi}
\partial_z \Phi(z,t)=1+\int_{0}^{t} \bigg(2\omega(\Phi(z,s)-\gamma_1,s)-\omega(\Phi(z,s)+\gamma_2,s)\bigg)ds.
\end{equation}
Definition of $\tau_1$ implies that $\partial_z \Phi(z,t)\geq 1$ for all $z\geq L_1,$ $t \leq \tau_1$. Denote $Z_1(t):=\Phi^{-1}(\Phi(L_1,t)+\gamma_2,t)$, then
$$\gamma_2=\Phi(Z_1(t),t)-\Phi(L_1,t)\geq Z_1(t)-L_1$$
which yields
\begin{equation} \label{boundForZ1}
Z_1(t) \leq L_1+\gamma_2.
\end{equation}

Using \eqref{BSinZ}, \eqref{alongTrajectories2}, $\rho_0\leq 1,$ the definition of $\tau_0$ \eqref{tau0con}, and $\gamma_2 \leq L_1/4$ we obtain that
\begin{align*}
\frac{d}{dt}\Phi(L_1,t)&\geq -\int_{\Phi(L_1,t)-\gamma_1}^{\Phi(L_1,t)+\gamma_2}\omega(y,t)dy\geq-\int_{\Phi(L_1,t)-\gamma_1}^{\Phi(L_1,t)+\gamma_2}\bigg( \int_{0}^{t} e^{\Phi(\Phi^{-1}(y,t),s)}ds\bigg)dy \\
                       &\geq -\int_{\Phi(L_1,t)-\gamma_1}^{\Phi(L_1,t)+\gamma_2}\bigg( \int_{0}^{t} e^{\Phi(Z_1(t),s)}ds\bigg)dy\geq -\int_{\Phi(L_1,t)-\gamma_1}^{\Phi(L_1,t)+\gamma_2}\bigg( \int_{0}^{t} e^{\Phi(L_1+\gamma_2,s)}ds\bigg)dy\\
                       &\geq -(\gamma_2+\gamma_1)te^{\Gamma(L_1+\gamma_2,t)} \geq -(\gamma_2+\gamma_1) \frac{\epsilon}{\tau_0(\gamma_1+\gamma_2+\epsilon)} \geq -\frac{\epsilon}{\tau_0},\quad \forall t\leq \tau_1
\end{align*}

Therefore, using our assumptions on $\epsilon,$ $\gamma_2,$ $L_0$ and \eqref{tau0con1}, we have for all $t \leq \tau_1$
$$\Phi(L_1,t)\geq  L_1-\epsilon \geq \gamma_1+\Phi(L_0,t)$$
which finishes the proof of the lemma.
\end{proof}

We write $z_{-}^1(\tau_1)=\Phi^{-1}(\Phi(z_1,\tau_1)-\gamma_1,\tau_1),$ $z_{+}^1(\tau_1)=\Phi^{-1}(\Phi(z_1,\tau_1)+\gamma_2,\tau_1)$, then naturally $z_{-}^1(\tau_1)<z_{+}^1(\tau_1)$ by non-intersection of trajectories. Let $0<s<\tau_1$ be such that
\begin{equation} \label{chooseS}
\Phi(z_{-}^1(\tau_1),s)\leq \Phi(z_{+}^1(\tau_1),s)-\gamma_1-\gamma_2-\epsilon
\end{equation}
Note that such $s$ must exist. Otherwise, Lemma~\ref{decaycon} guarantees that $\rho(z^1_{-}(t)) > 0$, and the breakthrough scenario at $\tau_1$ cannot happen
 due to (\ref{contra}) and \eqref{defEpsilon}. Let us focus on $s<\tau_1$ that is the maximal time for which equality in \eqref{chooseS} holds.
Now
\begin{align*}
\gamma_1+\gamma_2 &=(\Phi(z_1,\tau_1)+\gamma_2)-(\Phi(z_1,\tau_1)-\gamma_1)\\
                  &=\Phi[\Phi^{-1}(\Phi(z_1,\tau_1)+\gamma_2,\tau_1),\tau_1]-\Phi[\Phi^{-1}(\Phi(z_1,\tau_1)-\gamma_1,\tau_1),\tau_1]\\
                  &=\bigg(\Phi(z_{+}^1(\tau_1),s)+\int_{s}^{\tau_1} u(\Phi(z_{+}^1(\tau_1)),r)dr\bigg) - \bigg(\Phi(z_{-}^1(\tau_1),s) + \int_{s}^{\tau_1} u(\Phi(z_{-}^1(\tau_1)),r)dr\bigg)\\
                  &= \gamma_1+\gamma_2+\epsilon + \int_{s}^{\tau_1} dr\int _{\Phi(z_{+}^1(\tau_1),r)}^{\Phi(z_{-}^1(\tau_1),r)} \frac{\partial u}{\partial y}(y,r)dy\\
                  &= \gamma_1+\gamma_2+\epsilon + \int_{s}^{\tau_1} dr\int _{\Phi(z_{-}^1(\tau_1),r)}^{\Phi(z_{+}^1(\tau_1),r)} 2\omega(y-\gamma_1,r)-\omega(y+\gamma_2,r)dy.
\end{align*}
The choice of $s$ and \eqref{chooseS} immediately give $\Phi(z_{+}^1(\tau_1),r)-\Phi(z_{-}^1(\tau_1),r) < \gamma_1+\gamma_2+\epsilon$ for all $r\in(s,\tau_1)$, which implies that there must exist some $r\in(s,\tau_1)$ and some $y_0\in [\Phi(z_{-}(\tau_1),r),\Phi(z_{+}(\tau_1),r)]$ such that
\begin{equation} \label{contra2}
2\omega(y_0-\gamma_1,r)-\omega(y_0+\gamma_2,r)< -\frac{\epsilon}{(\tau_1-s)(\gamma_1+\gamma_2+\epsilon)}< -\frac{\epsilon}{\tau_0(\gamma_1+\gamma_2+\epsilon)}
\end{equation}
From the definition of $\tau_1$, we can infer that the only possibility is that $y_0=\Phi(z_0,r)$ for some $z_0\in[0,L_1)$.
Once (\ref{contra2}) is established, what is left to obtain contradiction is just to estimate $\omega(y_0+\gamma_2,r)$.
Note that by the second inequality in \eqref{decayest}, we have
$\Phi(z_0,r)+\gamma_2\leq \Phi(2L_1,r).$
Using this and (\ref{alongTrajectories2}) we get
\begin{align} \label{contra3}
\omega(y_0+\gamma_2,r)= \rho_0(\Phi^{-1}(y_0+\gamma_2,r))\int_{0}^{r} e^{\Phi(\Phi^{-1}(y_0+\gamma_2,r),r')}dr'\leq re^{\Gamma(2L_1,r)}< \tau_0e^{\Gamma(2L_1,\tau_0)}.
\end{align}
Non-negativity of $\omega$ together with (\ref{contra2}) and (\ref{contra3}) jointly contradict the choice of $\tau_0$ \eqref{tau0con} and the proof is complete.
\end{proof}

Let us reiterate that as opposed to Lemma \ref{positivityLemma1}, Proposition \ref{positivityProp1} holds for $\tau_0$ independent of $L_3$.
We are now free to choose $L_3$ large enough and assume \eqref{keyomexp1} for all times while solution exists.
The next proposition strengthens the bound in \eqref{keyomexp1} in a narrower range of $z.$

\begin{proposition} \label{positivityProp2}
Suppose $L_3 > L_2\geq L_1+\gamma_1$. Then for all $z\in[L_2,L_3]$ and $t\in[0,\tau_0]$, and while the regular solution exists, we have
\begin{equation}
2\omega(\Phi(z,t)-\gamma_1,t)-\omega(\Phi(z,t)+\gamma_2,t)\geq (2e^{-\gamma_1-\gamma_2}-1)\int_{0}^{t}e^{\Phi(z,s)}ds.
\end{equation}
\begin{proof}
First notice that for such choice of $L_2$, due to \eqref{dzPhi} and Proposition~\ref{positivityProp1}
we have $\Phi(L_2,t)\geq \Phi(L_1,t)+\gamma_1$ for all $t\in[0,\tau_0]$. This in turn ensures that $z_{-}(t)=\Phi^{-1}(\Phi(z,t)-\gamma_1,t)\geq L_1$ for all $z\geq L_2$. Then, based on Proposition \ref{positivityProp1} we obtain that
\begin{align}
\gamma_1+\gamma_2 &=(\Phi(z,t)+\gamma_2)-(\Phi(z,t)-\gamma_1) \nonumber\\
                  &=\bigg(\Phi(z_{+}(t),s)+\int_{s}^{t} u(\Phi(z_{+}(t),r),r)dr\bigg) - \bigg(\Phi(z_{-}(t),s) + \int_{s}^{t} u(\Phi(z_{-}(t),r),r)dr\bigg) \nonumber\\
                  &= \Phi(z_{+}(t),s)-\Phi(z_{-}(t),s)+ \int_{s}^{\tau_1} dr\int _{\Phi(z_{-}'(\tau_1),r)}^{\Phi(z_{+}'(\tau_1),r)} \bigg(2\omega(y-\gamma_1,r)-\omega(y+\gamma_2,r)\bigg)dy \nonumber \\
                  &\geq \Phi(z_{+}(t),s)-\Phi(z_{-}(t),s). \label{LeftRightToZ}
\end{align}
However, observe also that when $z\in[L_2,L_3]$, we always have $\rho_0(z_{-}(t))=1$. So we can recall (\ref{contra}) and combine with $(\ref{LeftRightToZ})$ to deduce that
\begin{eqnarray*}
2\omega(\Phi(z,t)-\gamma_1,t)-\omega(\Phi(z,t)+\gamma_2,t) \geq \int_{0}^{t}\bigg(2e^{\Phi(z_{-}(t),s)}-e^{\Phi(z_{+}(t),s)}\bigg)ds\\ \geq \int_{0}^{t}e^{\Phi(z_{+}(t),s)}(2e^{-\gamma_1-\gamma_2}-1)ds
      \geq (2e^{-\gamma_1-\gamma_2}-1)\int_{0}^{t}e^{\Phi(z,s)}ds.
\end{eqnarray*}
\end{proof}
\end{proposition}

Let us prove one last lemma before we move on to show finite time singularity formation.
\begin{lemma} \label{compareToF}
Define $f(z,t)$ for $z\in[L_2,L_3]$ by
\begin{align} \label{equationForF}
\partial_t f(z,t)= c\int_{0}^{t} e^{\int_{L_2}^{z} f(y,t)dy}dt,\quad f(z,0)\equiv 1/2,
\end{align}
where $c$ is a fixed positive constant.
Then\\
(a) For each $L_2 < L_3 <\infty$, the equation is locally well-posed;\\
(b) Given any $\tau_0>0,$ $L_3<\infty$ can be chosen so that $f(L_3,t)$ becomes infinite before $\tau_0$.
\end{lemma}
\begin{proof}
(a) Local existence of solutions can be done via a standard iteration argument. For an a-priori bound, set $h(t):=\sup_{z\in[L_2,L_3]} f(z,t)$. Then differentiating (\ref{equationForF}) gives
\begin{equation}\label{hub} \partial_{tt} h\leq ce^{h(t)(L_3-L_2)},\quad h(0)=1/2, \quad h'(0)=0, \end{equation}
which clearly controls $h(t)$ locally (we can use equality in \eqref{hub} to derive an upper bound). Suppose $h(t)$  stays bounded on $[0,T_0]$. Define the iteration scheme by
$$\partial_t f_n(t)=c\int_{0}^{t}e^{\int_{L_2}^{z}f_{n-1}(y,s)dy}ds, \quad f_n(z,0)=\frac12.$$
It can be seen by induction that $f_n(z,t)$ is increasing for every $z,t$. Note that $f_n$ is bounded by $h$ uniformly for all $n$, $z,$ and $t$. Thus, for $(z,t)\in [L_2,L_3]\times [0,T_0]$ we have
$$|\partial_t (f_n-f_{n-1})(z,t)|\leq c \int_{0}^{t}e^{\|h\|_{L^\infty}(z-L_2)}\max_{[L_2,L_3]\times [0,t]}|f_n(z,s)-f_{n-1}(z,s)|ds.$$
Let $F_n(t)=\max_{[L_2,L_3]\times [0,t]}|f_n(z,s)-f_{n-1}(z,s)|$, then $F(t)$ satisfies
$$F_n'(t)\leq C(L_3,T_0) \int_{0}^{t}F_{n-1}(s)ds$$
for some $C(L_3,T_0) < \infty$ which inductively gives
$$F_1(t)\leq \int_{0}^{t}ds\int_{0}^{s}e^{L_3/2}dr=e^{L_3/2} \frac{t^2}{2},\quad F_n(t)\leq \frac{C(L_3,T_0)^n t^{2n}}{(2n)!}$$
It is clear that the series $F_n$ converges uniformly in $z$ if $t \leq T_0.$ \\
(b) We will show now that $L_3$ can always be chosen so that $f(L_3,t)$ will go to infinity before $\tau_0$. First of all, note that from the definition we have $\partial_t f\geq 0$ for all $t\in[0,\tau_0]$
and $z \in [L_2,L_3]$  and hence for $z\in[L_3-(L_3-L_2)/2,L_3]$ and $t\geq \tau_0/2$
$$f(z,t)\geq \frac{1}{2}+c e^{\frac{L_3-L_2}{4}} \frac{\tau^2_0}{8}.$$
Denote $G_n:=(n+1)2^{n+2}$ and $\Delta=L_3-L_2$ and choose for $\Delta>8$ large enough so that $c e^{\Delta/4} \frac{\tau^2_0}{8} \geq G_1 \equiv 16$. We assert that in fact
\begin{equation} \label{induction}
f(z,\tau_0(1-2^{-n}))\geq G_n, \qquad \forall z\in [L_3-\Delta2^{-n},L_3].
\end{equation}
This assertion can be shown by an inductive argument. The case $n=1$ is instantaneous from the assumption on $\Delta$. Assume by induction that $(\ref{induction})$ holds for $n=k$. When $n=k+1$, for each $z\in [L_3-\Delta 2^{-k-1},L_3]$ we have
\begin{align*}
&f(z,\tau_0(1-2^{-k-1}))\geq \int_{\tau_0(1-2^{-k})}^{\tau_0(1-2^{-k-1})} dt\int_{\tau_0(1-2^{-k})}^{t}ds \exp\bigg(\int_{L_3-\Delta 2^{-k}}^{L_3-\Delta 2^{-k-1}} G_k dy\bigg)\\
&\geq c\frac{\tau^2_0}{2^{2k+3}} e^{2(k+1)\Delta}\geq \frac{e^{2(k+1)\Delta}}{2^{2k+3}} \geq 2^{10k} \geq (k+1)2^{2k+3} = G_{k+1}
\end{align*}
This shows that $\lim_{t\to\tau}f(L_3,t)=\infty$ for some $\tau \leq \tau_0$ as desired.
\end{proof}

Now, we are well-prepared to prove Theorem~\ref{NonDefiniteBlowUp}.
\begin{proof}
Take $\rho_0$ as described in the beginning of this section.
Choose $L_0,L_1$ as above. Let $\tau_0$ satisfy \eqref{tau0con}, \eqref{tau0con1}, and \eqref{tau0con2}.
Suppose that $L_2 \geq L_1+\gamma_1.$ Let $f$ satisfy \eqref{equationForF} with $c = 2e^{-\gamma_1-\gamma_2}-1.$
Choose $L_3$ so that the blow up time $T$ of $f(L_3,t)$ satisfies $T \leq \tau_0$. Fix $L_4>L_3$. Consider
\begin{equation}
\frac{\partial^2\Phi(z,t)}{\partial t\partial z}=\partial_z u(\Phi(z,t),t)=\partial_z \Phi(z,t)(2\omega(\Phi(z,t)-\gamma_1,t)-\omega(\Phi(z,t)+\gamma_2,t))
\end{equation}
with $\partial_{z} \Phi(z,0)\equiv 1$. By Proposition \ref{positivityProp1}, we see that for all $z\in[L_1,\infty)$ and $t\in [0,\tau_0]$,
\begin{equation} \label{monotonicityConclusion}
\frac{\partial^2\Phi(z,t)}{\partial t\partial z}\geq 0
\end{equation}
which indicates that $\partial_z\Phi(z,t)\geq 1$ for all $t$, $z$ in these ranges. Using (\ref{monotonicityConclusion}) and invoking Proposition \ref{positivityProp2}, we can infer that
\begin{equation}\label{diffIneqForPhiZ}
\frac{\partial^2 \Phi(z,t)}{\partial t\partial z} \geq c \partial_{z}\Phi(z,t)\cdot\int_{0}^{t}e^{\Phi(z,s)}ds \geq c \partial_{z}\Phi(z,t)\cdot\int_{0}^{t}e^{\int_{L_2}^{z}\partial_z \Phi(y,s)dy}ds
\end{equation}
for $z \in [L_2,L_3],$ $0 \leq t \leq \tau_0$ while solution exists.
However, it is not hard to establish the comparison $\partial_{z} \Phi(z,t)\geq f(z,t)$ for all $z\in[L_2,L_3]$ all the way until blow up time $T \leq \tau_0$.
Indeed, note that $\partial_z\Phi(z,0)\equiv 1> f(z,0)$. Suppose $T_1<T$ is the first time time when there exists $z_1\in[L_2,L_3]$ such that
$$\partial_z \Phi(z_1,t)=f(z_1,t).$$
However, by (\ref{diffIneqForPhiZ}), for every $s\leq T_1\leq T$ we have
$$\left. \frac{\partial^2 \Phi(z,t)}{\partial t\partial z} \right|_{(z,t)=(z_1,s)} \geq c \int_{0}^{s} e^{\int_{L_2}^{z_1}\partial_z\Phi(y,r)dy}dr \geq c\int_{0}^{s} e^{\int_{L_2}^{z_1}f(y,r)dy}dr= \left. \frac{d}{dt} f(z_1,t) \right|_{t=s}$$
which is a contradiction.

This argument shows that, unless singularity develops earlier in some other way required by Proposition~\ref{BKMcriterion},
$\partial_z \Phi(z,t)$ becomes infinite for some $t \leq \tau_0.$ However this implies that $\partial_z \tilde{u}(z,t)$ becomes infinite too, where we are returning to the $\tilde{u}$ notation for the velocity in $z$ representation \eqref{BSinZ} in order to avoid confusion. But we have
\[ \partial_z \tilde{u}(z,t) = -\partial_x u(x(z)) - \frac{u(x(z))}{x(z)}. \]
Therefore, it is not hard to see that blow up in $\partial_z \tilde{u}(z)$ forces blow up in either $\|\partial_x u\|_{L^\infty},$ or $\|\omega\|_{L^\infty},$
or $\delta(t)^{-1}.$ At this point we can invoke Proposition~\ref{BKMcriterion} which gives us a set of minimal conditions that must happen when singularity forms.
\end{proof}

{\bf Acknowledgement.} Both authors gratefully acknowledge partial support of the NSF-DMS grant 1712294. HY thanks Duke University for its hospitality.

\end{document}